\documentclass[twoside,12pt]{article}
\usepackage{amssymb}
\usepackage{amsfonts}
\usepackage{amsmath}

\input{tcilatex}

\begin{document}

\title{Lie algebra of an $n$-Lie algebra\\
\medskip }
\author{Basile Guy Richard BOSSOTO$^{(1)}$, Eug\`{e}ne OKASSA$^{(2)}$,
Mathias OMPORO$^{(3)}$ \and \medskip \\
Universit\'{e} Marien NGOUABI,\\
$^{(1),(2)}$Facult\'{e} des Sciences et Techniques,\\
D\'{e}partement de Math\'{e}matiques\\
e-mail: bossotob@yahoo.fr\\
e-mail: eugeneokassa@yahoo.fr\\
\medskip \\
$^{(3)}$Ecole Normale Sup\'{e}rieure\\
e-mail: mathompo@yahoo.fr\\
B.P.69 - Brazzaville- (Congo)}
\date{}
\maketitle

\begin{abstract}
We construct the Lie algebra of an $n$-Lie algebra and we also define\ the
notion of cohomology of an $n$-Lie algebra.
\end{abstract}



\section{Introduction}

The notion of $n$-Lie algebra over a commutative field $K$ with
characteristic $zero$, $n$ an integer $\geq $ $2$, introduced par Filippov 
\cite{fil2}, is a generalization of the notion of Lie algebra which
corresponds with the usual case when $n=2$.

When $n\geq 2$ is an integer and $K$ a commutative field, an $n$-Lie algebra
structure on a $K$-vector space $\mathcal{G~}$is due to the existence of a
skew-symmetric $n$-multilinear map%
\begin{equation*}
\left\{ ,...,\right\} :\mathcal{G}^{n}=\mathcal{G}\times \mathcal{G}\times
...\times \mathcal{G}\longrightarrow \mathcal{G},(x_{1},x_{2},...,x_{n})%
\longmapsto \left\{ x_{1},x_{2},...,x_{n}\right\} \text{,}
\end{equation*}%
\ \ \ \ such that%
\begin{align*}
& \left\{ x_{1},x_{2},...,x_{n-1},\left\{ y_{1},y_{2},...,y_{n}\right\}
\right\} \\
& =\dsum\limits_{i=1}^{n}\left\{ y_{1},y_{2},...,y_{i-1},\left\{
x_{1},x_{2},...,x_{n-1},y_{i}\right\} ,y_{i+1},...,y_{n}\right\}
\end{align*}%
for any $x_{1},x_{2},...,x_{n-1},y_{1},y_{2},...,y_{n}$ elements of $%
\mathcal{G}$.

The above identity is called Jacobi identity of the $n$-Lie algebra $%
\mathcal{G}$.

From this generalization, many authors, \cite{fil1}, \cite{kas}, \cite{wil},
extended the following notions: ideal of an $n$-Lie algebra, semi simple $n$%
-Lie algebra, nilpotent $n$- Lie algebra, solvable $n$-Lie algebra, Cartan
subalgebra of an $n$-Lie algebra, etc.

The main goal of this paper is to construct the Lie algebra structure from
an $n$-Lie algebra: so in this new context, we give definitions of ideal of
an $n$-Lie algebra, semi simple $n$-Lie algebra, nilpotent $n$-Lie algebra,
solvable $n$-Lie algebra, Cartan subalgebra of an $n$-Lie algebra. We also
define the cohomology of an $n$-Lie algebra.

In what follows, $K$ denotes a commutative field with characteristic $zero$, 
$\mathcal{G}$ an $n$-Lie algebra over $K$ with bracket $\left\{
,...,\right\} $ and finally $n\geq 2$ an integer.

\section{ $n$-Lie algebra structure}

We recall that, \cite{fil2}, for $n\geq 2$, an $n$-Lie algebra structure on
a $K$-vector space $\mathcal{G~}$ is due to the existence of a
skew-symmetric $n$-multilinear map%
\begin{equation*}
\left\{ ,...,\right\} :\mathcal{G}^{n}=\mathcal{G}\times \mathcal{G}\times
...\times \mathcal{G}\longrightarrow \mathcal{G},(x_{1},x_{2},...,x_{n})%
\longmapsto \left\{ x_{1},x_{2},...,x_{n}\right\} \text{,}
\end{equation*}%
\ \ \ \ such that%
\begin{align*}
& \left\{ x_{1},x_{2},...,x_{n-1},\left\{ y_{1},y_{2},...,y_{n}\right\}
\right\}  \\
& =\dsum\limits_{i=1}^{n}\left\{ y_{1},y_{2},...,y_{i-1},\left\{
x_{1},x_{2},...,x_{n-1},y_{i}\right\} ,y_{i+1},...,y_{n}\right\} 
\end{align*}%
for any $x_{1},x_{2},...,x_{n-1},y_{1},y_{2},...,y_{n}$ elements of $%
\mathcal{G}$.

A derivation of an $n$-Lie algebra $(\mathcal{G},\left\{ ,...,\right\} )$,
is a $K$-linear map%
\begin{equation*}
D:\mathcal{G}\longrightarrow \mathcal{G}
\end{equation*}%
such that%
\begin{equation*}
D\left\{ x_{1},x_{2},...,x_{n}\right\} =\dsum\limits_{i=1}^{n}\left\{
x_{1},x_{2},...,D(x_{i}),...,x_{n}\right\} 
\end{equation*}%
for any $x_{1},x_{2},...,x_{n}$ elements of $\mathcal{G}$.

We verify that the set of derivations of a $n$-Lie algebra $\mathcal{G}$ is
a Lie algebra over $K$ which we denote $Der_{K}(\mathcal{G})$.

\begin{proposition}
If $(\mathcal{G},\left\{ ,...,\right\} )$ is an $n$-Lie algebra, then for any
$x_{1},x_{2},...,x_{n-1}$ elements of $\mathcal{G}$ the map
\begin{equation*}
ad(x_{1},x_{2},...,x_{n-1}):\mathcal{G}\longrightarrow\mathcal{G}%
,y\longmapsto\left\{ x_{1},x_{2},...,x_{n-1},y\right\} \text{,}
\end{equation*}
is a derivation of $(\mathcal{G},\left\{ ,...,\right\} )$.
\end{proposition}

\begin{proof}
For any $x_{1},x_{2},...,x_{n-1}$ elements of $\mathcal{G}$ and for any $%
y_{1},y_{2},...,y_{n}$ elements of $\mathcal{G}$, we have
\begin{align*}
& \left[ ad(x_{1},x_{2},...,x_{n-1})\right] (\left\{
y_{1},y_{2},...,y_{n}\right\} ) \\
& =\left\{ x_{1},x_{2},...,x_{n-1},\left\{ y_{1},y_{2},...,y_{n}\right\}
\right\} \\
& =\dsum \limits_{i=1}^{n}\left\{ y_{1},y_{2},...,y_{i-1},\left\{
x_{1},x_{2},...,x_{n-1},y_{i}\right\} ,y_{i+1},...,y_{n}\right\} \\
& =\dsum \limits_{i=1}^{n}\left\{ y_{1},y_{2},...,y_{i-1},\left[
ad(x_{1},x_{2},...,x_{n-1})\right] (y_{i}),y_{i+1},...,y_{n}\right\} \text{.}
\end{align*}

And that ends the proof.
\end{proof}

A morphism of an $n$-Lie algebra $\mathcal{G}$ into another $n$-Lie algebra $%
\mathcal{G}^{\prime}$ is a $K$-linear map%
\begin{equation*}
\varphi:\mathcal{G}\longrightarrow\mathcal{G}^{\prime}
\end{equation*}
such that%
\begin{equation*}
\varphi(\left\{ x_{1},x_{2},...,x_{n}\right\} )=\left\{ \varphi
(x_{1}),\varphi(x_{2}),...,\varphi(x_{n})\right\}
\end{equation*}
for any $x_{1},x_{2},...,x_{n}$ elements of $\mathcal{G}$.

We verify that the set of $n$-Lie algebras over $K$ is a category.

\subsection{Lie algebra structure deduced from an $n$-Lie algebra}

When $\mathcal{G}$ is an $n$-Lie algebra and when $Der_{K}(\mathcal{G})$ is
the Lie algebra of $K$-derivations of $\mathcal{G}$, then the multilinear map%
\begin{equation*}
\mathcal{G}^{n-1}\longrightarrow Der_{K}(\mathcal{G}%
),(x_{1},x_{2},...,x_{n-1})\longmapsto ad(x_{1},x_{2},...,x_{n-1})\text{,}
\end{equation*}%
is skew-symmetric. If we denote $\Lambda _{K}^{n-1}(\mathcal{G})$, the $(n-1)
$-exterior power of the $K$-vector space $\mathcal{G}$, there exists an
unique $K$-linear map%
\begin{equation*}
ad_{\mathcal{G}}:\Lambda _{K}^{n-1}(\mathcal{G})\longrightarrow Der_{K}(%
\mathcal{G})
\end{equation*}%
such that%
\begin{equation*}
ad_{\mathcal{G}}(x_{1}\Lambda x_{2}\Lambda ...\Lambda
x_{n-1})=ad(x_{1},x_{2},...,x_{n-1})
\end{equation*}%
for any $x_{1},x_{2},...,x_{n-1}$ elements of $\mathcal{G}$.

We recall, \cite{bou}, that when 
\begin{equation*}
f:W\longrightarrow W
\end{equation*}%
is an endomorphism of a $K$-vector space $W$ and when $\Lambda _{K}(W)$ is
the $K$-exterior algebra of $W$, then there exists an unique derivation with
degree $zero$%
\begin{equation*}
D_{f}:\Lambda _{K}(W)\longrightarrow \Lambda _{K}(W)
\end{equation*}%
such that, for any $p\in 
\mathbb{N}
$,%
\begin{equation*}
D_{f}(w_{1}\Lambda w_{2}\Lambda ...\Lambda
w_{p})=\sum\limits_{i=1}^{p}w_{1}\Lambda w_{2}\Lambda ...\Lambda
w_{i-1}\Lambda f(w_{i})\Lambda w_{i+1}\Lambda ...\Lambda w_{p}
\end{equation*}%
for any $w_{1},w_{2},...,w_{p}$ elements of $W$.

When%
\begin{equation*}
g:W\longrightarrow W
\end{equation*}
is another endomorphism of the $K$-vector space $W$, then 
\begin{equation*}
\left[ D_{f},D_{g}\right] =D_{\left[ f,g\right] }\text{,}
\end{equation*}
where the bracket $\left[ ,\right] $ is the usual bracket of endomorphisms.

\begin{proposition}
For any $s_{1}$, $s_{2}$ elements of $\Lambda_{K}^{n-1}(\mathcal{G})$, then
we have simultaneously
\begin{equation*}
\left[ ad_{\mathcal{G}}(s_{1}),ad_{\mathcal{G}}(s_{2})\right] =ad_{\mathcal{G%
}}\text{ }\left( D_{ad_{\mathcal{G}}(s_{1})}(s_{2})\right)
\end{equation*}
and
\begin{equation*}
\left[ ad_{\mathcal{G}}(s_{1}),ad_{\mathcal{G}}(s_{2})\right] =ad_{\mathcal{G%
}}\text{ }\left( -D_{ad_{\mathcal{G}}(s_{2})}(s_{1})\right)
\end{equation*}
\end{proposition}

\begin{proof}
We prove for indecomposable elements. Let $s_{1}=x_{1}\Lambda
x_{2}\Lambda...\Lambda x_{n-1}$ and $s_{2}=y_{1}\Lambda
y_{2}\Lambda...\Lambda y_{n-1}$. For any $a\in\mathcal{G}$, we get
\begin{align*}
& (\left[ ad_{\mathcal{G}}(s_{1}),ad_{\mathcal{G}}(s_{2})\right] )(a) \\
& =\left\{ x_{1},x_{2},...,x_{n-1},\left\{ y_{1},y_{2},...,y_{n-1},a\right\}
\right\} \\
& -\left\{ y_{1},y_{2},...,y_{n-1},\left\{ x_{1},x_{2},...,x_{n-1},a\right\}
\right\} \\
& =\dsum \limits_{i=1}^{n-1}\left\{ y_{1},y_{2},...,y_{i-1},\left\{
x_{1},x_{2},...,x_{n-1},y_{i}\right\} ,y_{i+1},...,y_{n-1},a\right\} \\
& +\left\{ y_{1},y_{2},...,y_{n-1},\left\{ x_{1},x_{2},...,x_{n-1},a\right\}
\right\} \\
& -\left\{ y_{1},y_{2},...,y_{n-1},\left\{ x_{1},x_{2},...,x_{n-1},a\right\}
\right\} \\
& =\dsum \limits_{i=1}^{n-1}\left\{ y_{1},y_{2},...,y_{i-1},\left\{
x_{1},x_{2},...,x_{n-1},y_{i}\right\} ,y_{i+1},...,y_{n-1},a\right\} \\
& =\left[ ad_{\mathcal{G}}(\dsum \limits_{i=1}^{n-1}y_{1}\Lambda...\Lambda
y_{i-1}\Lambda\left[ ad_{\mathcal{G}}(x_{1}\Lambda x_{2}\Lambda...\Lambda
x_{n-1})\right] (y_{i})\Lambda y_{i+1}\Lambda ...\Lambda y_{n-1})\right] (a)%
\text{.}
\end{align*}

Thus we have
\begin{align*}
& \left[ ad_{\mathcal{G}}(s_{1}),ad_{\mathcal{G}}(s_{2})\right] \\
& =ad_{\mathcal{G}}(\dsum \limits_{i=1}^{n-1}y_{1}\Lambda
y_{2}\Lambda...\Lambda y_{i-1}\Lambda\left[ ad_{\mathcal{G}}(x_{1}\Lambda
x_{2}\Lambda...\Lambda x_{n-1})\right] (y_{i})\Lambda
y_{i+1}\Lambda...\Lambda y_{n-1}) \\
& =ad_{\mathcal{G}}\text{ }\left( D_{ad_{\mathcal{G}}(s_{1})}(s_{2})\right)
\text{.}
\end{align*}

On the other hand, we get
\begin{align*}
& (\left[ ad_{\mathcal{G}}(s_{1}),ad_{\mathcal{G}}(s_{2})\right] )(a) \\
& =\left\{ x_{1},x_{2},...,x_{n-1},\left\{ y_{1},y_{2},...,y_{n-1},a\right\}
\right\} \\
& -\left\{ y_{1},y_{2},...,y_{n-1},\left\{ x_{1},x_{2},...,x_{n-1},a\right\}
\right\} \\
& =\left\{ x_{1},x_{2},...,x_{n-1},\left\{ y_{1},y_{2},...,y_{n-1},a\right\}
\right\} \\
& -\dsum \limits_{i=1}^{n-1}\left\{ x_{1},x_{2},...,x_{i-1},\left\{
y_{1},y_{2},...,y_{n-1},x_{i}\right\} ,x_{i+1},...,x_{n-1},a\right\} \\
& -\left\{ x_{1},x_{2},...,x_{n-1},\left\{ y_{1},y_{2},...,y_{n-1},a\right\}
\right\} \\
& =-\dsum \limits_{i=1}^{n-1}\left\{ x_{1},x_{2},...,x_{i-1},\left\{
y_{1},y_{2},...,y_{n-1},x_{i}\right\} ,x_{i+1},...,x_{n-1},a\right\} \\
& =\left[ ad_{\mathcal{G}}(-\dsum \limits_{i=1}^{n-1}x_{1}\Lambda...\Lambda
x_{i-1}\Lambda\left[ ad_{\mathcal{G}}(y_{1}\Lambda...\Lambda y_{n-1})\right]
(x_{i})\Lambda x_{i+1}\Lambda...\Lambda x_{n-1})\right] (a)\text{.}
\end{align*}

Thus
\begin{equation*}
\left[ ad_{\mathcal{G}}(s_{1}),ad_{\mathcal{G}}(s_{2})\right] =ad_{\mathcal{G%
}}\text{ }\left( -D_{ad_{\mathcal{G}}(s_{2})}(s_{1})\right) \text{.}
\end{equation*}

That ends the proofs.
\end{proof}

We denote $\mathcal{V}_{K}(\mathcal{G})$ the $K$-subvector space of $%
\Lambda_{K}^{n-1}(\mathcal{G})$ generated by the elements of the form%
\begin{equation*}
D_{ad_{\mathcal{G}}(s_{1})}(s_{2})+D_{ad_{\mathcal{G}}(s_{2})}(s_{1})
\end{equation*}
where $s_{1}$ and $s_{2}$ describe $\Lambda_{K}^{n-1}(\mathcal{G})$.

Let 
\begin{equation*}
\Lambda_{K}^{n-1}(\mathcal{G})\longrightarrow\Lambda_{K}^{n-1}(\mathcal{G})/%
\mathcal{V}_{K}(\mathcal{G}),s\longmapsto\overline{s},
\end{equation*}
be the canonical surjection. Considering that precedes, we deduce that 
\begin{equation*}
ad_{\mathcal{G}}\left[ \mathcal{V}_{K}(\mathcal{G})\right] =0\text{.}
\end{equation*}
We denote 
\begin{equation*}
\widetilde{ad_{\mathcal{G}}}:\Lambda_{K}^{n-1}(\mathcal{G})/\mathcal{V}(%
\mathcal{G})\longrightarrow Der_{K}(\mathcal{G})
\end{equation*}
the unique linear map such that%
\begin{equation*}
\widetilde{ad_{\mathcal{G}}}(\overline{s})=ad_{\mathcal{G}}(s)
\end{equation*}
for any $s\in\Lambda_{K}^{n-1}(\mathcal{G})$.

For $s_{1},s_{1}^{\prime},s_{2},s_{2}^{\prime}$ elements of $%
\Lambda_{K}^{n-1}(\mathcal{G})$, we have%
\begin{align*}
& D_{ad_{\mathcal{G}}(s_{1})}(s_{2})+D_{ad_{\mathcal{G}}(s_{2}^{%
\prime})}(s_{1}^{\prime}) \\
& =D_{ad_{\mathcal{G}}(s_{1}-s_{1}^{\prime})}(s_{2})+D_{ad_{\mathcal{G}%
}(s_{2}^{\prime}-s_{2})}(s_{1}^{\prime})+D_{ad_{\mathcal{G}%
}(s_{1}^{\prime})}(s_{2})+D_{ad_{\mathcal{G}}(s_{2})}(s_{1}^{\prime})\text{.}
\end{align*}

We deduce that when%
\begin{equation*}
\overline{s_{1}}=\overline{s_{1}^{\prime}}
\end{equation*}
and 
\begin{equation*}
\overline{s_{2}}=\overline{s_{2}^{\prime}}\text{,}
\end{equation*}
then 
\begin{equation*}
D_{ad_{\mathcal{G}}(s_{1})}(s_{2})+D_{ad_{\mathcal{G}}(s_{2}^{%
\prime})}(s_{1}^{\prime})=D_{ad_{\mathcal{G}}(s_{1}^{\prime})}(s_{2})+D_{ad_{%
\mathcal{G}}(s_{2})}(s_{1}^{\prime})\text{.}
\end{equation*}

Finally we get%
\begin{equation*}
\overline{D_{ad_{\mathcal{G}}(s_{1})}(s_{2})}=\overline{D_{ad_{\mathcal{G}%
}(s_{1}^{\prime})}(s_{2}^{\prime})}\text{.}
\end{equation*}

Thus the bracket 
\begin{equation*}
\left[ \overline{s_{1}},\overline{s_{2}}\right] =\overline {D_{ad_{\mathcal{G%
}}(s_{1})}(s_{2})}
\end{equation*}
is well defined.

\begin{theorem}
When $(\mathcal{G},\left\{ ,...,\right\} )$ is an $n$-Lie algebra, then the
map
\begin{equation*}
\left[ ,\right] :\left[ \Lambda_{K}^{n-1}(\mathcal{G})/\mathcal{V}_{K}(%
\mathcal{G})\right] ^{2}\longrightarrow\Lambda_{K}^{n-1}(\mathcal{G})/%
\mathcal{V}_{K}(\mathcal{G}),(\overline{s_{1}},\overline{s_{2}})\longmapsto%
\overline{D_{ad_{\mathcal{G}}(s_{1})}(s_{2})},
\end{equation*}
only depends on $\overline{s_{1}}$ and $\overline{s_{2}}$, and defines a Lie
algebra structure on $\Lambda_{K}^{n-1}(\mathcal{G})/\mathcal{V}_{K}(%
\mathcal{G})$.

Moreover the map
\begin{equation*}
\widetilde{ad_{\mathcal{G}}}:\Lambda_{K}^{n-1}(\mathcal{G})/\mathcal{V}_{K}(%
\mathcal{G})\longrightarrow Der_{K}(\mathcal{G}),\overline{s}\longmapsto ad_{%
\mathcal{G}}(s),
\end{equation*}
is a morphism of $K$-Lie algebras.
\end{theorem}

\begin{proof}
The map $\left[ ,\right] $ is obviously bilinear. We have $\left[ \overline{%
s_{1}},\overline{s_{2}}\right] =\overline{D_{ad_{\mathcal{G}}(s_{1})}(s_{2})}
$. As $D_{ad_{\mathcal{G}}(s_{1})}(s_{2})+D_{ad_{\mathcal{G}}(s_{2})}(s_{1})$
$\in$ $\mathcal{V}_{K}(\mathcal{G})$, we immediately get%
\begin{equation*}
\overline{D_{ad_{\mathcal{G}}(s_{1})}(s_{2})}=-\overline{D_{ad_{\mathcal{G}%
}(s_{2})}(s_{1})}\text{.}
\end{equation*}

Thus $\left[ \overline{s_{1}},\overline{s_{2}}\right] =-\left[ \overline{%
s_{2}},\overline{s_{1}}\right] $.

For Jacobi identity, we write:%
\begin{align*}
& \left[ \overline{s_{1}},\left[ \overline{s_{2}},\overline{s_{3}}\right] %
\right] +\left[ \overline{s_{2}},\left[ \overline{s_{3}},\overline{s_{1}}%
\right] \right] +\left[ \overline{s_{3}},\left[ \overline{s_{1}},\overline{%
s_{2}}\right] \right] \\
& =\left[ \overline{s_{1}},\overline{D_{ad_{\mathcal{G}}(s_{2})}(s_{3})}%
\right] -\left[ \overline{s_{2}},\left[ \overline{s_{1}},\overline{s_{3}}%
\right] \right] +\left[ \overline{s_{3}},\overline{D_{ad_{\mathcal{G}%
}(s_{1})}(s_{2})}\right] \\
& =\left[ \overline{s_{1}},\overline{D_{ad_{\mathcal{G}}(s_{2})}(s_{3})}%
\right] -\left[ \overline{s_{2}},\overline{D_{ad_{\mathcal{G}}(s_{1})}(s_{3})%
}\right] -\left[ \overline{D_{ad_{\mathcal{G}}(s_{1})}(s_{2})},\overline{%
s_{3}}\right] \\
& =\overline{D_{ad_{\mathcal{G}}(s_{1})}\left[ D_{ad_{\mathcal{G}%
}(s_{2})}(s_{3})\right] }-\overline{D_{ad_{\mathcal{G}}(s_{2})}\left[ D_{ad_{%
\mathcal{G}}(s_{1})}(s_{3})\right] }-\overline{D_{ad_{\mathcal{G}}\left[
D_{ad_{\mathcal{G}}(s_{1})}(s_{2})\right] }(s_{3})} \\
& =\overline{\left[ D_{ad_{\mathcal{G}}(s_{1})},D_{ad_{\mathcal{G}}(s_{2})}%
\right] (s_{3})}-\overline{D_{\left[ ad_{\mathcal{G}}(s_{1}),ad_{\mathcal{G}%
}(s_{2})\right] }(s_{3})} \\
& =0\text{.}
\end{align*}

Moreover, we get
\begin{align*}
\left[ \widetilde{ad_{\mathcal{G}}}(\overline{s_{1}}),\widetilde {ad_{%
\mathcal{G}}}(\overline{s_{2}})\right] & =\left[ ad_{\mathcal{G}}(s_{1}),ad_{%
\mathcal{G}}(s_{2})\right] \\
& =ad_{\mathcal{G}}\text{ }\left( D_{ad_{\mathcal{G}}(s_{1})}(s_{2})\right)
\\
& =\widetilde{ad_{\mathcal{G}}}(\overline{D_{ad_{\mathcal{G}}(s_{1})}(s_{2})}%
) \\
& =\widetilde{ad_{\mathcal{G}}}(\left[ \overline{s_{1}},\overline{s_{2}}%
\right] )\text{.}
\end{align*}

That ends the proof of the two assertions.
\end{proof}

\begin{remark}
Thus the map%
\begin{equation*}
\widetilde{ad_{\mathcal{G}}}:\Lambda_{K}^{n-1}(\mathcal{G})/\mathcal{V}_{K}(%
\mathcal{G})\longrightarrow Der_{K}(\mathcal{G}),\overline{s}\longmapsto ad_{%
\mathcal{G}}(s),
\end{equation*}
is a representation of $\Lambda_{K}^{n-1}(\mathcal{G})/\mathcal{V}_{K}(%
\mathcal{G})$ into $\mathcal{G}$.
\end{remark}

\begin{proposition}
If $\mathcal{G}$ is an $n$-Lie algebra, then the space of invariant elements
of $\mathcal{G}$ for the representation%
\begin{equation*}
\widetilde{ad_{\mathcal{G}}}:\Lambda_{K}^{n-1}(\mathcal{G})/\mathcal{V}_{K}(%
\mathcal{G})\longrightarrow Der_{K}(\mathcal{G}),\overline{s}\longmapsto ad_{%
\mathcal{G}}(s),
\end{equation*}
is the following set
\begin{equation*}
Inv(\mathcal{G})=\left\{ x\in\mathcal{G}/\left\{ x,y_{1},y_{2},...,y_{n-1}=0%
\text{ }\right\} \right\}
\end{equation*}
for any $y_{1},y_{2},...,y_{n-1}\in\mathcal{G}$.
\end{proposition}

\begin{proof}
Considering the representation%
\begin{equation*}
\widetilde{ad_{\mathcal{G}}}:\Lambda_{K}^{n-1}(\mathcal{G})/\mathcal{V}_{K}(%
\mathcal{G})\longrightarrow Der_{K}(\mathcal{G}),\overline{s}\longmapsto ad_{%
\mathcal{G}}(s)\text{,}
\end{equation*}

we know that
\begin{equation*}
Inv(\mathcal{G})=\left\{ x\in\mathcal{G}/\left[ \widetilde{ad_{\mathcal{G}}}(%
\overline{s})\right] (x)=0\text{ }\right\}
\end{equation*}

for any $\overline{s}\in\Lambda_{K}^{n-1}(\mathcal{G})/\mathcal{V}_{K}(%
\mathcal{G})$.

We verify that
\begin{equation*}
Inv(\mathcal{G})=\left\{ x\in\mathcal{G}/\left\{ x,y_{1},y_{2},...,y_{n-1}=0%
\text{ }\right\} \right\}
\end{equation*}

for any $y_{1},y_{2},...,y_{n-1}\in\mathcal{G}$.
\end{proof}

\begin{proposition}
When $\mathcal{G}$ is an $n$-Lie algebra, a subspace $\mathcal{G}_{0}$ of $%
\mathcal{G}$ is stable for the representation

\begin{equation*}
\widetilde{ad_{\mathcal{G}}}:\Lambda_{K}^{n-1}(\mathcal{G})/\mathcal{V}_{K}(%
\mathcal{G})\longrightarrow Der_{K}(\mathcal{G}),\overline{s}\longmapsto ad_{%
\mathcal{G}}(s)\text{,}
\end{equation*}
if and only if for any $x\in$ $\mathcal{G}_{0}$ and for any $%
y_{1},y_{2},...,y_{n-1}\in\mathcal{G},$ we have
\begin{equation*}
\left\{ x,y_{1},y_{2},...,y_{n-1}\right\} \in\mathcal{G}_{0}\text{.}
\end{equation*}
\end{proposition}

\begin{proof}
It is obvious.
\end{proof}

In what follows, we give the relation between the category of $n$-Lie
algebras and the category of Lie algebras.

\begin{proposition}
The correspondence
\begin{equation*}
\mathfrak{L}^{n}:\mathcal{G}\longrightarrow\mathfrak{L}^{n}(\mathcal{G}%
)=\Lambda_{K}^{n-1}(\mathcal{G})/\mathcal{V}_{K}(\mathcal{G})
\end{equation*}
is a covariant functor from the category of $n$-Lie algebras to the category
of Lie algebras.
\end{proposition}

\begin{proof}
It is too quite obvious.
\end{proof}

When $F$ is a vector subspace of $\mathcal{G}$, we denote $%
\Lambda_{K}^{n-1}(F)$ the set of finite sums 
\begin{equation*}
\left\{ \underset{i_{1}<i_{2}<...<i_{n-1}}{\dsum }x_{i_{1}}\Lambda
x_{i_{2}}\Lambda...\Lambda x_{i_{n-1}}/i_{1},i_{2},...,i_{n-1}\in%
\mathbb{N}
,x_{i_{1}},x_{i_{2}},...,x_{i_{n-1}}\in F\right\} \text{.}
\end{equation*}

We constructed a Lie algebra from an $n$-Lie algebra. Considering the
functor 
\begin{equation*}
\mathfrak{L}^{n}:\mathcal{G}\longrightarrow \mathfrak{L}^{n}(\mathcal{G}%
)=\Lambda _{K}^{n-1}(\mathcal{G})/\mathcal{V}_{K}(\mathcal{G})\text{,}
\end{equation*}%
we will say that a subspace $I\subset $ $\mathcal{G}$ is an ideal of the $n$%
-Lie algebra $\mathcal{G}$ if the image of the space $\Lambda _{K}^{n-1}(I)$
by the canonical surjection 
\begin{equation*}
\Lambda _{K}^{n-1}(\mathcal{G})\longrightarrow \Lambda _{K}^{n-1}(\mathcal{G}%
)/\mathcal{V}_{K}(\mathcal{G}),s\longmapsto \overline{s},
\end{equation*}%
is an ideal of the Lie algebra $\Lambda _{K}^{n-1}(\mathcal{G})/\mathcal{V}%
_{K}(\mathcal{G})$.

We also will say that a subspace vectoriel $C\subset$ $\mathcal{G}$ is
Cartan subalgebra of the $n$-Lie algebra $\mathcal{G}$ if the image of the
space $\Lambda_{K}^{n-1}(C)$ by the canonical surjection 
\begin{equation*}
\Lambda_{K}^{n-1}(\mathcal{G})\longrightarrow\Lambda_{K}^{n-1}(\mathcal{G})/%
\mathcal{V}_{K}(\mathcal{G}),s\longmapsto\overline{s},
\end{equation*}
is a Cartan subalgebra of the Lie algebra $\Lambda_{K}^{n-1}(\mathcal{G})/%
\mathcal{V}_{K}(\mathcal{G})$.

We finally will say that an $n$-Lie algebra $\mathcal{G}$ is a semi simple $n
$-Lie algebra (nilpotent $n$-Lie algebra, solvable $n$-Lie algebra,
commutative $n$-Lie algebra respectively) if the Lie algebra $\Lambda
_{K}^{n-1}(\mathcal{G})/\mathcal{V}_{K}(\mathcal{G})$ is semi simple
(nilpotent, solvable, commutative respectively).

\begin{proposition}
If $Inv(\mathcal{G})$ is the space of invariants elements of $\mathcal{G}$
for the representation
\begin{equation*}
\widetilde{ad_{\mathcal{G}}}:\Lambda_{K}^{n-1}(\mathcal{G})/\mathcal{V}_{K}(%
\mathcal{G})\longrightarrow Der_{K}(\mathcal{G}),\overline{s}\longmapsto ad_{%
\mathcal{G}}(s)\text{,}
\end{equation*}
then the image of $\Lambda_{K}^{n-1}\left[ Inv(\mathcal{G})\right] $ by the
canonical surjection
\begin{equation*}
\Lambda_{K}^{n-1}(\mathcal{G})\longrightarrow\Lambda_{K}^{n-1}(\mathcal{G})/%
\mathcal{V}_{K}(\mathcal{G}),
\end{equation*}
is contained in the center of the Lie algebra $\Lambda_{K}^{n-1}(\mathcal{G}%
)/\mathcal{V}_{K}(\mathcal{G})$.
\end{proposition}

\begin{proof}
We reason with indecomposable elements. We consider $x_{1},x_{2},...,x_{n-1}$
elements of $Inv(\mathcal{G})$ and $y_{1},y_{2},...,y_{n-1}$ elements of $%
\mathcal{G}$. We get
\begin{align*}
& \left[ \overline{x_{1}\Lambda x_{2}\Lambda...\Lambda x_{n-1}},\overline{%
y_{1}\Lambda y_{2}\Lambda...\Lambda y_{n-1}}\right] \\
& =-\left[ \overline{y_{1}\Lambda y_{2}\Lambda...\Lambda y_{n-1}},\overline{%
x_{1}\Lambda x_{2}\Lambda...\Lambda x_{n-1}}\right] \\
& =-\overline{\dsum \limits_{i=1}^{n-1}x_{1}\Lambda x_{2}\Lambda...\Lambda
x_{i-1}\Lambda\left\{ y_{1},y_{2},...,y_{n-1},x_{i}\right\} \Lambda
x_{i+1}\Lambda...\Lambda x_{n-1}}.
\end{align*}

As%
\begin{equation*}
\left\{ y_{1},y_{2},...,y_{n-1},x_{i}\right\} =0
\end{equation*}

for $i=1,2,...,n-1$, then%
\begin{equation*}
\left[ \overline{x_{1}\Lambda x_{2}\Lambda...\Lambda x_{n-1}},\overline {%
y_{1}\Lambda y_{2}\Lambda...\Lambda y_{n-1}}\right] =0\text{.}
\end{equation*}

That ends the proof.
\end{proof}

We will establish a similar state for stable subspaces.

\begin{proposition}
If a subspace $\mathcal{G}_{0}$ of an $n$-Lie algebra $\mathcal{G}$ is stable
for the representation%
\begin{equation*}
\widetilde{ad_{\mathcal{G}}}:\Lambda_{K}^{n-1}(\mathcal{G})/\mathcal{V}_{K}(%
\mathcal{G})\longrightarrow Der_{K}(\mathcal{G}),\overline{s}\longmapsto ad_{%
\mathcal{G}}(s)\text{,}
\end{equation*}
then $\mathcal{G}_{0}$ is an ideal of the $n$-Lie algebra $\mathcal{G}$,
i.e. the image of $\Lambda_{K}^{n-1}(\mathcal{G}_{0})$ by the canonical
surjection
\begin{equation*}
\Lambda_{K}^{n-1}(\mathcal{G})\longrightarrow\Lambda_{K}^{n-1}(\mathcal{G})/%
\mathcal{V}_{K}(\mathcal{G}),
\end{equation*}
is an ideal of the Lie algebra $\Lambda_{K}^{n-1}(\mathcal{G})/\mathcal{V}%
_{K}(\mathcal{G})$.
\end{proposition}

\begin{proof}
Here, we also reason with indecomposable elements. If we consider $%
x_{1},x_{2},...,x_{n-1}$ elements of $\mathcal{G}_{0}$ and $%
y_{1},y_{2},...,y_{n-1}$ elements of $\mathcal{G}$. We get
\begin{align*}
& \left[ \overline{x_{1}\Lambda x_{2}\Lambda...\Lambda x_{n-1}},\overline{%
y_{1}\Lambda y_{2}\Lambda...\Lambda y_{n-1}}\right] \\
& =-\left[ \overline{y_{1}\Lambda y_{2}\Lambda...\Lambda y_{n-1}},\overline{%
x_{1}\Lambda x_{2}\Lambda...\Lambda x_{n-1}}\right] \\
& =-\overline{\dsum \limits_{i=1}^{n-1}x_{1}\Lambda x_{2}\Lambda...\Lambda
x_{i-1}\Lambda\left\{ y_{1},y_{2},...,y_{n-1},x_{i}\right\} \Lambda
x_{i+1}\Lambda...\Lambda x_{n-1}}.
\end{align*}

As%
\begin{equation*}
\left\{ y_{1},y_{2},...,y_{n-1},x_{i}\right\} \in\mathcal{G}_{0}
\end{equation*}

for $i=1,2,...,n-1$, then the bracket%
\begin{equation*}
\left[ \overline{x_{1}\Lambda x_{2}\Lambda...\Lambda x_{n-1}},\overline {%
y_{1}\Lambda y_{2}\Lambda...\Lambda y_{n-1}}\right]
\end{equation*}

belongs to the image of $\Lambda_{K}^{n-1}(\mathcal{G}_{0})$ by the
canonical surjection
\begin{equation*}
\Lambda_{K}^{n-1}(\mathcal{G})\longrightarrow\Lambda_{K}^{n-1}(\mathcal{G})/%
\mathcal{V}_{K}(\mathcal{G})\text{.}
\end{equation*}

That ends the proof.
\end{proof}

\subsection{Cohomology of an $n$-Lie algebra}

When $\mathcal{G}$ is an $n$-Lie algebra, we denote $d_{n}$ the cohomolgy
operator associated with the representation 
\begin{equation*}
\widetilde{ad_{\mathcal{G}}}:\Lambda _{K}^{n-1}(\mathcal{G})/\mathcal{V}_{K}(%
\mathcal{G})\longrightarrow Der_{K}(\mathcal{G}),\overline{s}\longmapsto ad_{%
\mathcal{G}}(s)\text{.}
\end{equation*}

For any $p\in%
\mathbb{N}
$, 
\begin{equation*}
\mathcal{L}_{sks}^{p}(\Lambda_{K}^{n-1}(\mathcal{G})/\mathcal{V}_{K}(%
\mathcal{G}),\mathcal{G})
\end{equation*}
denotes the $K$-vector space of skew-symmetric $p$-multilinear maps of

\begin{equation*}
\Lambda _{K}^{n-1}(\mathcal{G})/\mathcal{V}_{K}(\mathcal{G})
\end{equation*}%
into $\mathcal{G}$ and 
\begin{equation*}
\mathcal{L}_{sks}(\left[ \Lambda _{K}^{n-1}(\mathcal{G})/\mathcal{V}_{K}(%
\mathcal{G})\right] ,\mathcal{G})=\bigoplus\limits_{p\in 
\mathbb{N}
}\mathcal{L}_{sks}^{p}(\left[ \Lambda _{K}^{n-1}(\mathcal{G})/\mathcal{V}%
_{K}(\mathcal{G})\right] ,\mathcal{G})\text{.}
\end{equation*}%
We will say that the cohomology of the differential complex 
\begin{equation*}
\left( \mathcal{L}_{sks}(\left[ \Lambda _{K}^{n-1}(\mathcal{G})/\mathcal{V}%
_{K}(\mathcal{G})\right] ,\mathcal{G}),d_{n}\right)
\end{equation*}%
is the cohomology of the $n$-Lie algebra $\mathcal{G}$.

We denote 
\begin{equation*}
H_{n}(\mathcal{G})=Ker(d_{n})/\func{Im}(d_{n})\text{.}
\end{equation*}

\begin{proposition}
When $\mathcal{G}$ is an $n$-Lie algebra, then
\begin{equation*}
H_{n}^{0}(\mathcal{G})=Inv(\mathcal{G})\text{.}
\end{equation*}
\end{proposition}

\end{document}